\documentclass[12pt]{amsart}
\usepackage{amsmath, amsthm, amssymb}
\usepackage[top=1.25in, bottom=1.25in, left=1.0in, right=1.0in]{geometry}

\theoremstyle{plain}
\newtheorem{thm}{Theorem}

\newtheorem{lem}[thm]{Lemma}
\newtheorem{cor}[thm]{Corollary}
\newtheorem*{OreBrooks}{Ore Version of Brooks' Theorem}
\newtheorem*{BrooksTheorem}{Brooks' Theorem}
\newtheorem*{MainTheorem}{Main Theorem}

\theoremstyle{definition}
\newtheorem{defn}{Definition}

\theoremstyle{remark}

\newtheorem*{observation}{Observation}
\newtheorem*{TemptingThought}{Tempting Thought}

\newcommand{\fancy}[1]{\mathcal{#1}}

\newcommand{\card}[1]{\left|#1\right|}
\newcommand{\size}[1]{\left\Vert#1\right\Vert}
\newcommand{\ceil}[1]{\left\lceil#1\right\rceil}
\newcommand{\floor}[1]{\left\lfloor#1\right\rfloor}

\newcommand{\parens}[1]{\left( #1 \right)}

\title{An improvement on Brooks' Theorem}
\author{Landon Rabern}

\begin{document}
\begin{abstract}
We prove that $\chi(G) \leq \max \left\{\omega(G), \Delta_2(G), \frac{5}{6}(\Delta(G) + 1)\right\}$ for every graph $G$ with $\Delta(G) \geq 3$.  Here $\Delta_2$ is the parameter introduced by Stacho that gives the largest degree that a vertex $v$ can have subject to the condition that $v$ is adjacent to a vertex whose degree is at least as large as its own.  This upper bound generalizes both Brooks' Theorem and the Ore-degree version of Brooks' Theorem.  
\end{abstract}
\maketitle
\section{Introduction}
Brooks' Theorem \cite{Brooks} gives an upper bound on a graph's chromatic number in terms of its maximum degree and clique number.  

\begin{BrooksTheorem}
Every graph with $\Delta \geq 3$ satisfies $\chi \leq \max\{\omega, \Delta\}$.
\end{BrooksTheorem}

In \cite{Stacho} Stacho introduced the graph parameter $\Delta_2$ as the largest degree that a vertex $v$ can have subject to the condition that $v$ is adjacent to a vertex whose degree is at least as large as its own.  He proved that for any graph $G$, the bound $\chi(G) \leq \Delta_2(G) + 1$ holds. Moreover, he proved that for any fixed $t \geq 3$, the problem of determining whether or not $\chi(G) \leq \Delta_2(G)$ for graphs with $\Delta_2(G) = t$ is \emph{NP}-complete.  It is tempting to think that an analogue of Brooks' Theorem like the following holds for $\Delta_2$.

\begin{TemptingThought}
There exists $t$ such that every graph with $\Delta_2 \geq t$ satisfies $\chi \leq \max \{\omega, \Delta_2\}$.
\end{TemptingThought}

Unfortunately, using Lov\'{a}sz's $\vartheta$ parameter \cite{Lovasz} which can be computed in polynomial time and has the property that $\omega(G) \leq \vartheta(\overline{G}) \leq \chi(G)$ we see immediately that if \emph{P} $\neq$ \emph{NP}, then the tempting thought cannot hold for any $t$.  In the final section we give a construction showing that this is indeed the case whether or not \emph{P} $\neq$ \emph{NP}. However, if we limit how far from $\Delta + 1$ our upper bound can stray, we can get a generalization of Brooks' Theorem involving $\Delta_2$.

\begin{MainTheorem}
Every graph with $\Delta \geq 3$ satisfies \[\chi \leq \max \left\{\omega, \Delta_2, \frac{5}{6}(\Delta + 1)\right\}.\]
\end{MainTheorem}

In addition to generalizing Brooks' Theorem, this also generalizes the Ore-degree version of Brooks' Theorem as introduced by Kierstead and Kostochka in \cite{KK} and improved in \cite{DeltaCritical}.

\begin{defn}
The \emph{Ore-degree} of an edge $xy$ in a graph $G$ is $\theta(xy) = d(x) + d(y)$.  The \emph{Ore-degree} of a graph $G$ is $\theta(G) = \max_{xy \in E(G)}\theta(xy)$.
\end{defn}

Note that $\Delta_2 \leq \left\lfloor\frac{\theta}{2} \right \rfloor \leq \Delta$.  In \cite{DeltaCritical} the following bound was proved. The graph  $O_5$ exhibited in \cite{KK} shows that the $\theta \geq 10$ condition is best possible.

\begin{OreBrooks}
Every graph with $\theta \geq 10$ satisfies $\chi \leq \max \left\{\omega, \left\lfloor\frac{\theta}{2} \right \rfloor\right\}$.
\end{OreBrooks}
\begin{proof}
Suppose the theorem is false and choose a counterexample $G$ minimizing $\card{G}$.  Plainly, $G$ is vertex critical.  Thus $\delta(G) \geq \chi(G) - 1$.  In particular, $\theta(G) \geq \delta(G) + \Delta(G) \geq \chi(G) + \Delta(G) - 1$.  Hence $\Delta(G) \leq \chi(G)$.  Applying the Main Theorem, we conclude $\Delta(G) \leq \chi(G) \leq \frac{5}{6}(\Delta(G) + 1)$ and hence $\Delta(G) \leq 5$.  But then $\theta(G) = 10$ and we must have $\chi(G) \geq 6$.  Now applying Brooks' Theorem gets the desired contradiction.
\end{proof}

\noindent In fact, a similar proof shows that a whole spectrum of generalizations hold.

\begin{defn}
For $0 \leq \epsilon \leq 1$, define $\Delta_\epsilon(G)$ as

\[\left\lfloor\max_{xy \in E(G)} (1 - \epsilon)\min\{d(x), d(y)\} + \epsilon\max\{d(x), d(y)\}\right\rfloor.\]
\end{defn}

\noindent  Note that $\Delta_1 = \Delta$, $\Delta_{\frac12} = \left\lfloor\frac{\theta}{2} \right\rfloor$ and $\Delta_0 = \Delta_2$.

\begin{thm}\label{TheoremR}
For every $0 < \epsilon \leq 1$, there exists $t_\epsilon$ such that every graph with $\Delta_\epsilon \geq t_\epsilon$ satisfies 

\[\chi \leq \max\{\omega, \Delta_\epsilon\}.\]
\end{thm}

\noindent It would be interesting to determine, for each $\epsilon$, the smallest $t_\epsilon$ that works in Theorem \ref{TheoremR}. In the final section we give a simple construction showing that $t_\epsilon \geq 1 + \frac{2}{\epsilon}$.  The Main Theorem implies $t_\epsilon < \frac{6}{\epsilon}$.

\section{Rephrasing the problem}

\begin{defn}
For a graph $G$ and $r \geq 0$, let $G^{\geq r}$ be the subgraph of $G$ induced on the vertices of degree at least $r$ in $G$.  Let $\mathcal{H}(G) = G^{\geq \chi(G)}$.
\end{defn}

\noindent We can rewrite the definition of $\Delta_2$ as

\[\Delta_2(G) = \min \left\{r \geq 0 \mid G^{\geq r} \text{ is edgeless} \right\} - 1.\]

\noindent In particular we have the following.
\begin{observation}
For any graph $G$, $\chi(G) > \Delta_2(G)$ if and only if $\mathcal{H}(G)$ is edgeless.
\end{observation}

\noindent This observation will allow us to prove our upper bound without worrying about $\Delta_2$.

\section{Proving the bound}
We will use part of an algorithm of Mozhan \cite{Mozhan}.  The following is a generalization of his main lemma.

\begin{defn}
Let $G$ be a graph containing at least one critical vertex.  Let $a \geq 1$ and $r_1, \ldots, r_a$ be such that $1 + \sum_i r_i = \chi(G)$. By a \emph{$(r_1, \ldots, r_a)$-partitioned coloring} of $G$ we mean a proper coloring of $G$ of the form
\[\left\{\{x\}, L_{11}, L_{12}, \ldots, L_{1r_1}, L_{21}, L_{22}, \ldots, L_{2r_2}, \ldots, L_{a1}, L_{a2}, \ldots, L_{ar_a}\right\}.\]

\noindent Here $\{x\}$ is a singleton color class and each $L_{ij}$ is a color class.
\end{defn}

\begin{lem}\label{mozlemma}
Let $G$ be a graph containing at least one critical vertex.   Let $a \geq 1$ and $r_1, \ldots, r_a$ be such that $1 + \sum_i r_i = \chi(G)$. Of all $(r_1, \ldots, r_a)$-partitioned colorings of $G$ pick one (call it $\pi$) minimizing

\[\sum_{i = 1}^a \size{G\left[\bigcup_{j = 1}^{r_i} L_{ij}\right]}.\]

\noindent Remember that $\{x\}$ is a singleton color class in the coloring. Put $U_i = \bigcup_{j = 1}^{r_i} L_{ij}$ and let $Z_i(x)$ be the component of $x$ in $G[\{x\} \cup U_i]$.  If $d_{Z_i(x)}(x) = r_i$, then $Z_i(x)$ is complete if $r_i \geq 3$ and $Z_i(x)$ is an odd cycle if $r_i = 2$.
\end{lem}
\begin{proof}
Let $1 \leq i \leq a$ such that $d_{Z_i(x)}(x) = r_i$. Put $Z_i = Z_i(x)$.

First suppose that $\Delta(Z_i) > r_i$.  Take $y \in V(Z_i)$ with $d_{Z_i}(y) > r_i$ closest to $x$ and let $x_1x_2\cdots x_t$ be a shortest $x-y$ path in $Z_i$.  Plainly, for $k < t$, each $x_k$ hits exactly one vertex in each color class besides its own.  Thus we may recolor $x_k$ with $\pi(x_{k + 1})$ for $k < t$ and $x_t$ with $\pi(x_1)$ to produce a new $\chi(G)$-coloring of $G$ (this can be seen as a generalized Kempe chain).  But we've moved a vertex ($x_t$) of degree $r_i + 1$ out of $U_i$ while moving in a vertex ($x_1$) of degree $r_i$ violating the minimality condition on $\pi$.  This is a contradiction.

Thus $\Delta(Z_i) \leq r_i$.  But $\chi(Z_i) = r_i + 1$, so Brooks' Theorem implies that $Z_i$ is complete if $r_i \geq 3$ and $Z_i$ is an odd cycle if $r_i = 2$.
\end{proof}

\begin{defn}
We call $v \in V(G)$ \emph{low} if $d(v) = \chi(G) - 1$ and \emph{high} otherwise.
\end{defn}

Note that in Lemma \ref{mozlemma}, if $d_{Z_i(x)}(x) = r_i$ then we can \emph{swap} $x$ with any other $y \in Z_i(x)$ by changing $\pi$ so that $x$ is colored with $\pi(y)$ and $y$ is colored with $\pi(x)$ to get another minimal $\chi(G)$-coloring of $G$. 

\begin{lem}\label{JoinedLows}
Assume the same setup as Lemma \ref{mozlemma} and that $x$ is low. If $i \neq j$ such that $r_i \geq r_j \geq 3$ and a low vertex $w \in U_i \cap N(x)$ is adjacent to a low vertex $z \in U_j \cap N(x)$, then the low vertices in $(U_i \cup U_j) \cap N(x)$ are all universal in
$G[(U_i \cup U_j) \cap N(x)]$. 
\end{lem}
\begin{proof}
Suppose $i \neq j$ and a low vertex $w \in U_i \cap N(x)$ is adjacent to a low vertex $z \in U_j \cap N(x)$.  Swap $x$ with $w$ to get a new minimal $\chi(G)$-coloring of $G$.  Since $w$ is low and adjacent to $z \in U_j \cap N(x)$, $w$ is joined to $U_j \cap N(x)$ by Lemma \ref{mozlemma}.  Similarly $z$ is joined to $U_i \cap N(x)$.  But now every low vertex in $U_i \cap N(x)$ is adjacent to the low vertex $z \in U_j \cap N(x)$ and is hence joined to $U_j \cap N(x)$. Similarly, every low vertex in $U_j \cap N(x)$ is joined to $U_i \cap N(x)$.  Since both $U_i \cap N(x)$ and $U_j \cap N(x)$ induce cliques in $G$, the proof is complete.
\end{proof}

\begin{thm}\label{TheoremP}
Fix $k \geq 2$ and let $G$ be a vertex critical graph with $\chi(G) \geq \Delta(G) + 1 - k$.  If $\Delta(G) + 1 \geq 6k$ and $\mathcal{H}(G)$ is edgeless then $G = K_{\chi(G)}$.
\end{thm}
\begin{proof}
Suppose that $\Delta(G) + 1 \geq 6k$ and $\mathcal{H}(G)$ is edgeless. Since $\Delta(G) + 1 \geq 6k$ we have $\chi(G) \geq 5k$ and thus we can find $r_1, \ldots, r_{k+1}$ such that $r_1, r_2 \geq k + 1$, $r_i \geq 3$ for each $i \geq 3$ and $\sum_{i = 1}^{k+1} r_i = \chi(G) - 1$.  Note that $r_i \geq 3$ for each $i$ since $k \geq 2$.

Put $a = k+1$. Of all $(r_1, r_2, \ldots, r_a)$-partitioned colorings of $G$, pick one (call it $\pi$) minimizing

\[\sum_{i = 1}^a \size{G\left[\bigcup_{j = 1}^{r_i} L_{ij}\right]}.\]

Remember that $\{x\}$ is a singleton color class in the coloring. Throughout the proof we refer to a coloring that minimizes the above function as a \emph{minimal} coloring. Put $U_i = \bigcup_{j = 1}^{r_i} L_{ij}$ and let $C_i = \pi(U_i)$ (the colors used on $U_i$).  For a minimal coloring $\gamma$ of $G$, let $Z_{\gamma, i}(x)$ be the component of $x$ in $G[\{x\} \cup \gamma^{-1}(C_i)]$.  Note that $Z_i(x) = Z_{\pi, i}(x)$.

First suppose $x$ is high. Since $a > k$ we have $1 \leq i \leq a$ such that $d_{Z_i(x)}(x) = r_i$.  Thus $Z_i(x)$ is complete.  Since $\mathcal{H}(G)$ is edgeless, each vertex in $Z_i(x) - x$ must be low.  Hence we can swap $x$ with a low vertex in $U_i$ to get another minimal $\chi(G)$ coloring. Thus we may assume that $x$ is low.  Consider the following algorithm.
\begin{enumerate}
\item Put $q_0(y) = 0$ for each $y \in V(G)$.
\item Put $x_0 = x$, $\pi_0 = \pi$, $p_0 = 1$ and $i = 0$.
\item Pick a low vertex $x_{i + 1} \in Z_{\pi_i, p_i}(x_i) - x_i$ minimizing $q_i(x_{i + 1})$. Swap $x_{i + 1}$ with $x_i$. Let $\pi_{i+1}$ be the resulting coloring.
\item If  there exists $d \in  \{3, \ldots, a\} - \{p_i\}$ with $\left|V(Z_{\pi_{i + 1}, d}(x_{i + 1})) \cap \bigcup_{j = 1}^i x_j\right| = 0$, then let $p_{i+1} = d$.  Otherwise pick $p_{i+1} \in \{1,2\} - \{p_i\}$.
\item Put $q_i(x_i) = q_i(x_{i+1}) + 1$.
\item Put $q_{i+1} = q_i$.
\item Put $i = i + 1$.
\item Goto (3).
\end{enumerate}  

Since $G$ is finite we have a smallest $t$ such that for $p = 1$ or $p = 2$ with $p \neq p_{t-1}$ we have $\left|\left\{y \in V(Z_{\pi_t, p}(x_t)) - \{x_t\} \mid q_t(y) = 1\right\}\right| = k$. Let $x_{t_1}, \ldots, x_{t_k}$  with $t_1 < t_2 \cdots < t_k$ be the vertices in $V(Z_{\pi_t, p}(x_t)) - \{x_t\}$ with $q_t(x_{t_j}) = 1$.

Swap $x_t$ with $x_{t_1}$ and note that $x_{t_1}$ is low and adjacent to each of $x_{t_1 + 1}, \ldots, x_{t_k + 1}$.  Also note that $\{x_{t_1 + 1}, \ldots, x_{t_k + 1}\}$ induces a clique in $G$ since all those vertices are in $U_p$. By the condition in step (4) we see that $\{p_{t_1 + 1}, p_{t_2 + 1}, \ldots, p_{t_k + 1}\} = \{1, \ldots, a\} - \{p\}$.  Thus the low vertices in $\bigcup_{i \neq p} \pi_t^{-1}(C_i) \cap N(x_{t_1})$ are universal in $G\left[\bigcup_{i \neq p} \pi_t^{-1}(C_i) \cap N(x_{t_1})\right]$ by Lemma \ref{JoinedLows}.  Also since $x_t$ is low and is joined to $\pi_t^{-1}(C_i)  \cap N(x_{t_1})$ for each $i \neq p$, again applying Lemma \ref{JoinedLows} we get that the low vertices in $N(x_{t_1}) \cup \{x_{t_1}\}$ are universal in $G[N(x_{t_1}) \cup \{x_{t_1}\}]$.

Put $F = G[N(x_{t_1}) \cup \{x_{t_1}\}]$ and let $S$ be the set of high vertices in $F$.  Note that $|F| = \chi(G)$ and $|S| \leq k + 1$ since $\mathcal{H}(G)$ is edgeless.  We will show that $F$ is complete.  It will be enough to show that $S$ is a clique.  Suppose we have non-adjacent $w, z \in S$.  Color $G - F$ with $\chi(G) - 1$ colors.  This leaves a list assignment $L$ on $F$ with $|L(v)| \geq d_F(v) - k$ for each $v \in V(F)$.  Thus $|L(w)| + |L(z)| \geq d_F(w) + d_F(z) - 2k \geq 2(|F| - |S|) - 2k \geq 2 (\Delta(G) - 2k) - 2k = 2\Delta(G) - 6k$.  Since $\Delta(G) + 1 \geq 6k$ and $k \geq 2$, we have $|L(w)| + |L(z)| \geq 2\Delta(G) - 6k \geq \Delta(G) + 1 - k$.  Hence we have $c \in L(w) \cap L(z)$.  Color both $w$ and $z$ with $c$ to get a new list assignment $L'$ on $F' = F - \{w, z\}$.  Put $A = G[S - \{w,z\}]$. Then we can complete the coloring to $A$ since for any $v \in V(A)$ we have $|L'(v)| \geq d_{F'}(v) - k \geq d_A(v) + |F| - |S| - k \geq d_A(v) + \Delta(G) - 3k \geq d_A(v) + 1$.  Let $J$ be the resulting list assignment on $B = F - S$.  Since the vertices in $B$ are all low and they each have a pair of neighbors that received the same color ($w$ and $z$) we have $|J(v)| \geq d_B(v) + 1$ for each $v \in V(B)$.  Hence we can complete the $\chi(G) - 1$ coloring to all of $F$.  This is a contradiction.  Hence $S$ is a clique and the theorem is proved.
\end{proof}

The $k = 1$ case was dealt with in \cite{DeltaCritical}.  The proof is similar but complicated by having to deal with odd cycles instead of just cliques. There the following was proved.

\begin{cor}\label{CorN}
$K_{\chi(G)}$ is the only critical graph $G$ with $\chi(G) \geq \Delta(G) \geq 6$ such that $\mathcal{H}(G)$ is edgeless.
\end{cor}

\noindent Now the proof of the Main Theorem is almost immediate.

\begin{proof}[Proof of Main Theorem]
Suppose the theorem is false and choose a counterexample $G$ minimizing $\card{G}$.  Plainly, $G$ is vertex critical.  Let $k = \Delta(G) + 1 - \chi(G)$.  Note that $k \geq 1$ by Brooks' Theorem. Since $\chi(G) > \Delta_2(G)$, we know by our observation above that $\mathcal{H}(G)$ is edgeless.  Also, since $\chi(G) > \frac{5}{6}(\Delta(G) + 1)$ we have $\Delta(G) + 1 - k = \chi(G) \geq 5k + 1$.  If $k \geq 2$ we have a contradiction by Theorem \ref{TheoremP}.  If $k = 1$ we have a contradiction by Corollary \ref{CorN}. 
\end{proof}

\section{A simple construction}
Let $F_n$ be the graph formed from the disjoint union of $K_n - xy$ and $K_{n-1}$ by joining $\floor{\frac{n-1}{2}}$ vertices of the $K_{n-1}$ to $x$ and the other $\ceil{\frac{n-1}{2}}$ vertices of the $K_{n-1}$ to $y$.  It is easily verified that for $n \geq 4$ we have $\chi(F_n) = n > \omega(F_n)$, $\Delta(F_n) = \ceil{\frac{n-1}{2}} + n - 2$ and $\fancy{H}(G)$ is edgeless (and nonempty).  Moreover, $\Delta_\epsilon(F_n) = \floor{(1-\epsilon)(n-1) + \epsilon\parens{\ceil{\frac{n-1}{2}} + n - 2}} = \floor{n - 1 - \epsilon + \epsilon\ceil{\frac{n-1}{2}}}$.  For $0 < \epsilon \leq 1$, choose $n_\epsilon \in \mathbb{N}$ maximal such that $\ceil{\frac{n_\epsilon - 1}{2}} < 1 + \frac{1}{\epsilon}$.  Then $\Delta_\epsilon(F_{n_\epsilon}) = n_\epsilon - 1$.  Hence in Theorem \ref{TheoremR}, we must have $t_\epsilon \geq n_\epsilon$. By maximality, $n_\epsilon$ must be odd.  Thus

\[n_\epsilon = \begin{cases}
1 + \frac{2}{\epsilon} & \text{if } \frac{1}{\epsilon} \in \mathbb{N} \\
3 + 2\floor{\frac{1}{\epsilon}} & \text{if } \frac{1}{\epsilon} \not\in \mathbb{N}.
\end{cases}\]

\noindent In particular, $t_\epsilon \geq n_\epsilon \geq 1 + \frac{2}{\epsilon}$ for all $0 < \epsilon \leq 1$. Additionally, we see that $t_0$ does not exist; that is, the tempting thought is false.


\begin{thebibliography}{1}
\bibitem{Brooks}
R.L. Brooks. 
\newblock On colouring the nodes of a network.
\newblock{\em Math. Proc. Cambridge Philos. Soc.,} \textbf{37}, 1941, \mbox{194-197.}

\bibitem{Lovasz}
M. Gr\"{o}tschel, L. Lov\'{a}sz, and A. Schrijver.
\newblock The ellipsoid method and its consequences in combinatorial optimization.
\newblock{\em Combinatorica,} \textbf{1}, 1981, \mbox{169-197}.

\bibitem{KK}
H.A. Kierstead, A.V. Kostochka.
\newblock Ore-type versions of Brooks' theorem.
\newblock{\em Journal of Combinatorial Theory, Series B,} \textbf{99}, 2009, \mbox{298-305.}

\bibitem{Mozhan}
N.N. Mozhan.
\newblock Chromatic number of graphs with a density that does not exceed two-thirds of the maximal degree.
\newblock{\em Metody Diskretn. Anal.,} \textbf{39}, 1983, \mbox{52-65.}

\bibitem{DeltaCritical}
L. Rabern.
\newblock $\Delta$-Critical graphs with small high vertex cliques.
\newblock Journal of Combinatorial Theory Series B, In Press.

\bibitem{Stacho}
L. Stacho.
\newblock New Upper Bounds for the Chromatic Number of a Graph.
\newblock{\em Journal of Graph Theory,} \textbf{36(2)}, 2001, \mbox{117-120.}

\end{thebibliography}
\end{document}